\documentclass[a4paper,11pt,fleqn]{article}
\usepackage{enumitem}
\usepackage{amsmath}
\usepackage{amssymb}
\usepackage{pifont}
\usepackage{theorem} 
\usepackage{euscript}
\usepackage{exscale,relsize}
\usepackage{epic,eepic}
\usepackage{multicol}
\usepackage{makeidx}
\usepackage{srcltx}         
\usepackage{xcolor}
\usepackage{url}
\usepackage{charter}
\usepackage{mathrsfs} 
\usepackage{graphicx}
\usepackage{authblk}
\newcommand{\email}[1]{\href{mailto:#1}{\nolinkurl{#1}}}
\usepackage[normalem]{ulem}     
\usepackage[usenames,dvipsnames]{pstricks}
\usepackage{pstricks-add}
\usepackage{pst-grad}
\usepackage{pst-plot} 
\usepackage{pst-node} 
\usepackage{pst-eucl} 
\usepackage{pst-3dplot} 
\newlength{\mySubFigSize}
\definecolor{labelkey}{rgb}{0,0.08,0.45}
\definecolor{refkey}{rgb}{0,0.6,0.0}
\definecolor{Brown}{rgb}{0.45,0.0,0.05}
\definecolor{dgreen}{rgb}{0.00,0.49,0.00}
\definecolor{dblue}{rgb}{0,0.08,0.75}
\RequirePackage[dvips,colorlinks,hyperindex]{hyperref}
\hypersetup{linktocpage=true,citecolor=dblue,linkcolor=dgreen}
\renewcommand{\leq}{\ensuremath{\leqslant}}
\renewcommand{\geq}{\ensuremath{\geqslant}}

\renewcommand{\geq}{\ensuremath{\geqslant}}

\newcommand{\Frac}[2]{\displaystyle{\frac{#1}{#2}}} 
 
\newcommand{\scal}[2]{{\langle{{#1}\mid{#2}}\rangle}}

\newcommand{\menge}[2]{\big\{{#1}~\big |~{#2}\big\}}

\newcommand{\HH}{\ensuremath{{\mathcal H}}}
\newcommand{\GG}{\ensuremath{{\mathcal G}}}

\newcommand{\Sum}{\ensuremath{\displaystyle\sum}}
\newcommand{\emp}{\ensuremath{{\varnothing}}}

\newcommand{\Id}{\ensuremath{\operatorname{Id}}\,}

\newcommand{\RR}{\ensuremath{\mathbb{R}}}
\newcommand{\RP}{\ensuremath{\left[0,+\infty\right[}}

\newcommand{\RPP}{\ensuremath{\left]0,+\infty\right[}}

\newcommand{\spts}{\ensuremath{\text{\rm{spts}\,}}}
\newcommand{\barc}{\ensuremath{\text{\rm{bar}\,}}}

\newcommand{\RPX}{\ensuremath{\left[0,+\infty\right]}}
\newcommand{\RX}{\ensuremath{\left]-\infty,+\infty\right]}}

\newcommand{\EE}{\ensuremath{\mathsf E}}
\newcommand{\PP}{\ensuremath{\mathsf P}}

\newcommand{\NN}{\ensuremath{\mathbb N}}

\newcommand{\intdom}{\ensuremath{\text{int\,dom}\,}}

\newcommand{\ran}{\ensuremath{\text{\rm ran}\,}}

\newcommand{\pinf}{\ensuremath{{+\infty}}}
\newcommand{\minf}{\ensuremath{{-\infty}}}
\newcommand{\epi}{\ensuremath{\text{\rm epi}\,}}
\newcommand{\dom}{\ensuremath{\text{\rm dom}\,}}
\newcommand{\env}{\ensuremath{\text{\rm env}\,}}
\newcommand{\rec}{\ensuremath{\text{\rm rec}\,}}

\newcommand{\cone}{\ensuremath{\text{\rm cone}\,}}

\newcommand{\infconv}{\ensuremath{\mbox{\small$\,\square\,$}}}

\newtheorem{theorem}{Theorem}[section]
\newtheorem{lemma}[theorem]{Lemma}
\newtheorem{corollary}[theorem]{Corollary}
\newtheorem{proposition}[theorem]{Proposition}

\theoremstyle{plain}{\theorembodyfont{\rmfamily}%
}
\theoremstyle{plain}{\theorembodyfont{\rmfamily}%
\newtheorem{example}[theorem]{Example}}
\theoremstyle{plain}{\theorembodyfont{\rmfamily}%
\newtheorem{remark}[theorem]{Remark}}
\theoremstyle{plain}{\theorembodyfont{\rmfamily}%
}
\theoremstyle{plain}{\theorembodyfont{\rmfamily}%
}
\theoremstyle{plain}{\theorembodyfont{\rmfamily}%
}
\theoremstyle{plain}{\theorembodyfont{\rmfamily}%
\newtheorem{definition}[theorem]{Definition}}
\theoremstyle{plain}{\theorembodyfont{\rmfamily}%
}

\numberwithin{equation}{section}
\begin{document}

\title{\sffamily\LARGE Perspective Functions: Properties,
Constructions, and Examples\footnote{Contact author: 
P. L. Combettes, \email{plc@math.ncsu.edu}, 
phone: +1 (919) 515 2671.}}

\author{Patrick L. Combettes\\
\small
\small North Carolina State University\\
\small Department of Mathematics\\
\small Raleigh, NC 27695-8205, USA\\
\small \email{plc@math.ncsu.edu}\\
}

\date{\sffamily ~}

\maketitle

\vskip 8mm

\begin{abstract}
Many functions encountered in applied mathematics and in
statistical data analysis can be expressed in terms of 
perspective functions. One of the earliest examples is the 
Fisher information, which appeared in statistics in 
the 1920s. We analyze various algebraic and 
convex-analytical properties of perspective functions and 
provide general schemes to construct lower semicontinuous convex
functions from them. Several new examples are presented and 
existing instances are featured as special cases.
\end{abstract}

\section{Introduction}
Let $\GG$ be a real Hilbert space and let 
$\varphi\colon\GG\to\RX$ be a convex function.
The perspective function of $\varphi$ is (see Figure~\ref{fig:24})
\begin{equation}
\label{ekJhy64989d08}
\mathscr{P}_{\varphi}\colon\RR\times\GG\to\RX\colon 
(\eta,y)\mapsto
\begin{cases}
\eta\varphi(y/\eta),&\text{if}\;\:\eta>0;\\
\pinf,&\text{otherwise.}
\end{cases}
\end{equation}
The properties of $\mathscr{P}_{\varphi}$ were first investigated 
in \cite{Rock70}, where it was shown in particular that
$\mathscr{P}_{\varphi}$ is convex if and only if $\varphi$ is
convex (see also \cite{Livre1,Daco08,Hiri93}). The term ``perspective
function'' was coined by Claude Lemar\'echal ca.\ 1987-1988
\cite{Lema17} and first appeared in print in 
\cite[Section~IV.2.2]{Hiri93}.
Special cases of the construction \eqref{ekJhy64989d08} arise in 
various areas of applied mathematics and data analysis.
One of the oldest instances involving perspective functions is 
the Fisher information of a differentiable probability density 
$x\colon\RR^N\to\RPP$, that is,
\begin{equation}
\label{e:fisher}
\int_{\RR^N}\frac{\|\nabla x(t)\|_2^2}{x(t)}dt,
\end{equation}
where $\|\cdot\|_2$ is the standard Euclidean norm on $\RR^N$.
This notion, which dates back to the work of Fisher in statistics
\cite{Fish25}, has found applications in many contexts,
e.g., \cite{Berc13,Borw95,Borw96,Frie07,Noll98,Tosc15,Vill98}. 
More generally, \eqref{ekJhy64989d08} 
can be used to construct convex integrands of integral
functionals such as  
\begin{equation}
\label{ekJhy64989d21b}
\int_{\RR^N}\mathscr{P}_{\varphi}\big(x(t),y(t)\big)dt=
\int_{\RR^N}x(t)\varphi\bigg(\frac{y(t)}{x(t)}\bigg)dt,
\end{equation}
where $x\colon\RR^N\to\RPP$ and $y\colon\RR^N\to\GG$.
In the case when $N=1$ and $\GG=\RR$, it corresponds to a notion
of $\varphi$-divergence which originates in
\cite{Alis66,Csis67} and that has been used extensively in 
information theory, statistics, signal processing, and pattern 
recognition \cite{Bass89,Berl11,Lies06,Pard06}; see also 
\cite{Tebo91,Hiri07} for a 
discussion of discrete counterparts.
In the case when $\GG=\RR^{N\times N}$, $y=\nabla x$, and 
$\varphi=\|\cdot\|_2^2$, one recovers \eqref{e:fisher}.
Furthermore, choosing $\varphi=\|\cdot\|_2^p$ with 
$p\in\left]1,\pinf\right[$ provides the extension of the 
Fisher information \eqref{e:fisher} found in \cite{Boek77} in 
the case when $N=1$.
Instances of perspective functions can also be identified in 
robust estimation \cite[Section~7.7]{Hube09} (see also 
\cite{Ndia16,Owen07} for recent developments), 
transportation theory \cite{Bena00,Bras11,Fits16,Papa14}, sparse
regression \cite{Bien2016,Scda16b,Lederer2015}, control theory
\cite{Jung13,Moeh15}, mixed-integer programming \cite{Hija12},
computer vision \cite{Zach10}, disjunctive programming \cite{Ceri99},
game theory \cite{Akia16}, machine learning \cite{Micc13},
and mean-field games \cite{Bric16}.

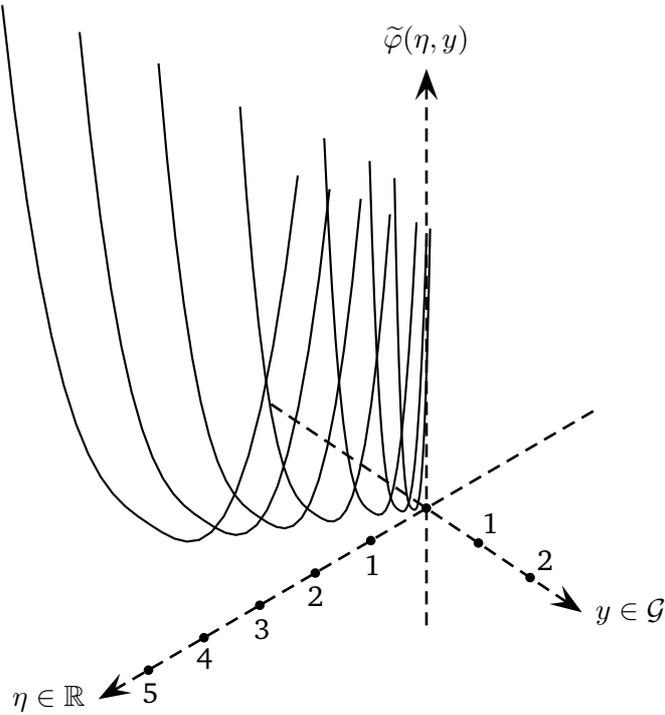
\begin{figure}
\begin{center}
\begin{pspicture}(-5,-2.6)(5,5.0)
\psset{Alpha=43,Beta=39}
\pstThreeDCoor[linestyle=dashed,nameX=$\eta\in\RR$,nameY=$y\in\GG$,
nameZ={$\widetilde{\varphi}(\eta,y)$},linecolor=black,%
arrowsize=0.30cm,linewidth=1.0pt,xMin=-3.0,xMax=5.9,%
yMin=-3,yMax=3,zMin=-2,zMax=7.5]
\parametricplotThreeD[algebraic,linecolor=black](-0.349,0.339)%
{0.25|t|0.125+abs(2*t)^3/.0625}
\parametricplotThreeD[algebraic,linecolor=black](-0.56,0.54)%
{0.5|t|0.25+abs(2*t)^3/.25}
\parametricplotThreeD[algebraic,linecolor=black](-0.9,0.88)%
{1|t|0.5+abs(2*t)^3/1}
\parametricplotThreeD[algebraic,linecolor=black](-1.45,1.44)%
{2|t|1.0+abs(2*t)^3/4}
\parametricplotThreeD[algebraic,linecolor=black](-1.95,1.95)%
{3|t|1.5+abs(2*t)^3/9}
\parametricplotThreeD[algebraic,linecolor=black](-2.4,2.42)%
{4|t|2.0+abs(2*t)^3/16}
\parametricplotThreeD[algebraic,linecolor=black](-2.82,2.88)%
{5|t|2.5+abs(2*t)^3/25}
\pstThreeDDot[dotscale=1,linecolor=black](1.0,0,0)
\pstThreeDDot[dotscale=1,linecolor=black](2.0,0,0)
\pstThreeDDot[dotscale=1,linecolor=black](3.0,0,0)
\pstThreeDDot[dotscale=1,linecolor=black](4.0,0,0)
\pstThreeDDot[dotscale=1,linecolor=black](5.0,0,0)
\pstThreeDDot[dotscale=1,linecolor=black](0,2,0)
\pstThreeDDot[dotscale=1,linecolor=black](0,1,0)
\pstThreeDDot[dotscale=1,linecolor=black](0,0,0)
\rput(-0.71,-0.76){1}
\rput(-1.46,-1.17){2}
\rput(-2.16,-1.6){3}
\rput(-2.92,-2.0){4}
\rput(-3.63,-2.45){5}
\rput(0.84,-0.2){1}
\rput(1.56,-0.69){2}
\end{pspicture}
\end{center}
\caption{Slices of the graph of the perspective function of 
$\varphi\colon y\mapsto 1/2+8\|y\|^3$ for fixed values of 
$\eta\in\{1/4,1/2,1,2,3,4,5\}$; the value
$\eta=1$ provides the graph of $\varphi$.}
\label{fig:24}
\end{figure}

Although perspective functions appear explicitly or implicitly in 
an increasing number of
diverse research areas, little effort has been dedicated to the
systematic study of their properties, especially in general
Hilbert spaces. It is the goal of the present paper to propose such
an investigation, with a special focus on the construction of lower
semicontinuous convex functions around perspective functions.
As is well known, these two properties are of paramount importance
in the modeling, analysis, and numerical solution of variational
problems. Section~\ref{sec:2} focuses on algebraic and 
convex-analytical properties. On the basis of these results,
several examples of lower semicontinuous convex perspective 
functions are provided in Section~\ref{sec:3}. 
Finally, integral functions with
perspective function-based integrands are studied in 
Section~\ref{sec:5}. 
Many of the functions we propose are
new and suggest new problem formulations in various applications
areas. In particular, our results are exploited in the 
companion paper \cite{Scda16b}, which investigates the proximity 
operator of perspective functions and explores new models and
algorithms in high-dimensional statistics. 

{\bfseries Notation.} 
Throughout, $\HH$ and $\GG$ are real Hilbert spaces and  
$\HH\oplus\GG$ denotes their Hilbert direct sum.
The closed ball with center $x\in\HH$ and radius $\rho\in\RPP$ in
$\HH$ is denoted by $B(x;\rho)$. 
$\Gamma_0(\HH)$ is the class of lower 
semicontinuous convex functions $f\colon\HH\to\RX$ such that 
$\dom f=\menge{x\in\HH}{f(x)<\pinf}\neq\emp$. Let
$f\in\Gamma_0(\HH)$. Then $f^*$ denotes the conjugate of $f$, 
$\epi f$ the epigraph of $f$, 
$\rec f$ the recession function of $f$, and $\partial f$ the
subdifferential of $f$.
Let $C$ be a subset of $\HH$. Then $\iota_C$ is the
indicator function of $C$, $d_C$ the distance function to $C$,
$\rec C$ the recession cone of $C$, and $\sigma_C$ the 
support function of $C$. See \cite{Livre1,Laur72} for background 
on hilbertian convex analysis and \cite{Hiri93,Rock70} for 
the Euclidean setting. 

\section{Properties of perspective functions}
\label{sec:2}

In this section we study various properties of perspective
functions. We start our discussion by noting that, if 
$\varphi\in\Gamma_0(\GG)$, the construction 
\eqref{ekJhy64989d08} does not necessarily produce a 
lower semicontinuous function. For this
reason, we shall use the following variant, first proposed in
\cite{Rock70} for $\GG=\RR^N$.

\begin{definition}
\label{d:perspective}
Let $\varphi\in\Gamma_0(\GG)$ and let $\rec\varphi$ be its 
recession function, i.e., given any $z\in\dom\varphi$,
\begin{equation}
\label{ebcBdq739jU9-23a}
(\forall y\in\GG)\quad(\rec\varphi)(y)=\sup_{x\in\dom
\varphi}\big(\varphi(x+y)-\varphi(y)\big)=\lim_{\alpha\to\pinf}
\frac{\varphi(z+\alpha y)}{\alpha}.
\end{equation}
The lower semicontinuous envelope of the perspective of $\varphi$ 
is
\begin{equation}
\label{ekJhy64989d08l}
\widetilde{\varphi}\colon\RR\times\GG\to\RX\colon 
(\eta,y)\mapsto
\begin{cases}
\eta\varphi(y/\eta),&\text{if}\;\:\eta>0;\\
(\rec\varphi)(y),&\text{if}\;\:\eta=0;\\
\pinf,&\text{otherwise.}
\end{cases}
\end{equation}
For simplicity, $\widetilde\varphi$ is called the
\emph{perspective} of $\varphi$.
\end{definition}

\begin{lemma}
\label{l:pjlaurent}
Let $\varphi\in\Gamma_0(\GG)$. Then the following hold:
\begin{enumerate}
\item
\label{l:pjlaurenti}
$\rec\epi\varphi=\epi\rec\varphi$ 
{\rm\cite[Proposition~6.8.3]{Laur72}}.
\item
\label{l:pjlaurentii}
$\rec\varphi=\sigma_{\dom\varphi^*}$
{\rm\cite[Th\'eor\`eme~6.8.5]{Laur72}}.
\end{enumerate}
\end{lemma}

The following result records basic topological and convex
analytical properties of the perspective function 
\eqref{ekJhy64989d08l}.

\begin{proposition}
\label{pkJhy64989d21}
Let $\varphi\in\Gamma_0(\GG)$. Then the following hold:
\begin{enumerate}
\item
\label{pkJhy64989d21i}
$\widetilde{\varphi}$ is positively homogeneous.
\item
\label{pkJhy64989d21ii}
$\widetilde{\varphi}\in\Gamma_0(\RR\oplus\GG)$. 
\item
\label{pkJhy64989d21ii+}
$\widetilde{\varphi}$ is sublinear.
\item
\label{pkJhy64989d21iii}
Let $C=\menge{(\mu,u)\in\RR\times\GG}{\mu+\varphi^*(u)\leq 0}$. 
Then $(\widetilde{\varphi})^*=\iota_C$ and 
$\widetilde{\varphi}=\sigma_C$.
\item
\label{pkJhy64989d21iii+}
Let $\eta\in\RR$ and $y\in\GG$. Then
\begin{equation}
\label{eIidfwetf548b9-23x}
\partial\widetilde{\varphi}(\eta,y)=
\begin{cases}
\menge{\big(\varphi(y/\eta)-\scal{y}{u}/\eta,u\big)}
{u\in\partial\varphi(y/\eta)},&\text{if}\;\:\eta>0;\\
\menge{(\mu,u)\in C}{\sigma_{\dom\varphi^*}(y)=\scal{y}{u}},
&\text{if}\;\:\eta=0\;\;\text{and}\;\;y\neq 0;\\
C,&\text{if}\;\:\eta=0\;\;\text{and}\;\;y=0;\\
\emp,&\text{if}\;\:\eta<0.
\end{cases}
\end{equation}
\end{enumerate}
\end{proposition}
\begin{proof}
\ref{pkJhy64989d21i}: This follows from \eqref{ebcBdq739jU9-23a} 
and \eqref{ekJhy64989d08l}.

\ref{pkJhy64989d21ii}: 
Set $D=\{1\}\times\epi\varphi$ and $g=\mathscr{P}_{\varphi}$,
and let $z\in\dom\varphi$. Then $(1,z)\in\dom g$. On the other 
hand, since $D$ is convex, $\epi g=\cone D$ is convex and 
$g$ is therefore a proper convex function. 
Let us denote by $\breve{g}$ the
largest lower semicontinuous convex function majorized by $g$. 
To show that $\widetilde{\varphi}\in\Gamma_0(\RR\oplus\GG)$, 
it is enough to show that 
\begin{equation}
\label{ekJhy64989d22v}
\widetilde{\varphi}=\breve{g}. 
\end{equation}
This can be done using the following argument due to H. H. 
Bauschke. Since $(0,0)\notin D$, it follows from 
\cite[Theorem~9.9 and Corollary~6.52]{Livre1},
Lemma~\ref{l:pjlaurent}\ref{l:pjlaurenti}, and
\cite[Lemma~1.6(ii)]{Livre1} that 
$\epi\breve{g}=\overline{\epi}f=\overline{\text{cone}}\,D=
(\cone D)\cup(\rec D)=
(\epi f)\cup(\{0\}\times\rec\epi\varphi)=
(\epi f)\cup(\{0\}\times\epi\rec\varphi)=
(\epi f)\cup\epi(\iota_{\{0\}}\oplus\rec\varphi)=
\epi\text{min}\{f,\iota_{\{0\}}\oplus\rec\varphi\}=
\epi\widetilde{\varphi}$.

\ref{pkJhy64989d21ii+}: This follows from 
\ref{pkJhy64989d21i} and \ref{pkJhy64989d21ii}.

\ref{pkJhy64989d21iii}: Set $g=\mathscr{P}_{\varphi}$. Then 
$g^*=\iota_C$ \cite[Example~13.8]{Livre1}. Hence, we derive from 
\eqref{ekJhy64989d22v} and \cite[Proposition~13.14]{Livre1} that
$(\widetilde{\varphi})^*=(\breve{g})^*=g^*=\iota_C$. In turn,
\ref{pkJhy64989d21ii} and \cite[Corollary~13.33]{Livre1} yield
$\widetilde{\varphi}=(\widetilde{\varphi})^{**}=\iota_C^*=\sigma_C$.

\ref{pkJhy64989d21iii+}: 
Let $\mu\in\RR$ and $u\in\GG$.
It follows from the Fenchel-Young identity 
\cite[Proposition~16.13]{Livre1} and
\ref{pkJhy64989d21iii} that 
\begin{eqnarray}
\label{eIidfwetf548b9-22}
(\mu,u)\in\partial\widetilde{\varphi}(\eta,y)
&\Leftrightarrow&\widetilde{\varphi}(\eta,y)+
(\widetilde{\varphi})^*(\mu,u)=\eta\mu+\scal{y}{u}\nonumber\\
&\Leftrightarrow&
\widetilde{\varphi}(\eta,y)=\eta\mu+\scal{y}{u}\;\;\text{and}\;\;
\mu+\varphi^*(u)\leq 0.
\end{eqnarray}
We consider three cases.
\begin{itemize}
\item
$\eta<0$:
Then \eqref{ekJhy64989d08l} and \eqref{eIidfwetf548b9-22} yield
$\partial\widetilde{\varphi}(\eta,y)=\emp$. 
\item
$\eta=0$: We deduce from
\eqref{eIidfwetf548b9-22}, \eqref{ekJhy64989d08l}, and 
Lemma~\ref{l:pjlaurent}\ref{l:pjlaurentii} that
\begin{eqnarray}
(\mu,u)\in\partial\widetilde{\varphi}(\eta,y)
&\Leftrightarrow&
(\rec\varphi)(y)=\scal{y}{u}\;\;\text{and}\;\;\mu+\varphi^*(u)\leq 0
\nonumber\\
&\Leftrightarrow&
\sigma_{\dom\varphi^*}(y)=\scal{y}{u}\;\;
\text{and}\;\;(\mu,u)\in C.
\end{eqnarray}
Since $\sigma_{\dom\varphi^*}(0)=0=\scal{0}{u}$, we obtain the
desired results.
\item
$\eta>0$:  Using successively
\eqref{eIidfwetf548b9-22}, \eqref{ekJhy64989d08l}, the
Fenchel-Young inequality \cite[Proposition~13.13]{Livre1},
and the Fenchel-Young identity, we obtain
\begin{eqnarray}
(\mu,u)\in\partial\widetilde{\varphi}(\eta,y)
&\Leftrightarrow&
\mu=\varphi(y/\eta)-\scal{y}{u}/\eta\;\;\text{and}\;\;
\varphi(y/\eta)+\varphi^*(u)\leq\scal{y/\eta}{u}\nonumber\\
&\Leftrightarrow&
\mu=\varphi(y/\eta)-\scal{y}{u}/\eta\;\;\text{and}\;\;
\varphi(y/\eta)+\varphi^*(u)=\scal{y/\eta}{u}\nonumber\\
&\Leftrightarrow&
\mu=\varphi(y/\eta)-\scal{y}{u}/\eta\;\;\text{and}\;\;
u\in\partial\varphi(y/\eta).
\end{eqnarray}
\end{itemize}
We have thus proved \eqref{eIidfwetf548b9-23x}.
\end{proof}

\begin{remark}
\label{rIidfwetf548b9-23}
Some of the results of Proposition~\ref{pkJhy64989d21} have
already been obtained in the case when $\GG=\RR^N$ with different
tools, some of which are specific to the finite-dimensional
setting. Thus, items \ref{pkJhy64989d21ii} and
\ref{pkJhy64989d21iii} can be found in \cite{Rock70}, and 
the case $\eta>0$ of \ref{pkJhy64989d21iii+} appears in 
\cite[Proposition~4]{Ceri99}.
\end{remark}

As shown in \cite{Scda16b}, \eqref{eIidfwetf548b9-23x} is
instrumental in computing the proximity operator of a perspective
function. Here is an important refinement. 

\begin{corollary}
\label{cIidfwetf548b9-23}
Let $\varphi\in\Gamma_0(\GG)$ and denote by $\barc\dom\varphi^*$ 
the barrier cone of $\dom\varphi^*$. Let $\eta\in\RR$, let 
$y\in\GG$, and suppose that one of the following holds:
\begin{enumerate}
\item
\label{cIidfwetf548b9-23i}
$y\notin\barc\dom\varphi^*$.
\item
\label{cIidfwetf548b9-23ii}
$\dom\varphi^*$ is open.
\item
\label{cIidfwetf548b9-23iii}
$\dom\varphi^*=\GG$.
\item
\label{cIidfwetf548b9-23iv}
$\varphi$ is supercoercive: 
$\lim_{\|y\|\to\pinf}\varphi(y)/\|y\|=\pinf$.
\item
\label{cIidfwetf548b9-23v}
For every $v\in\GG$, $\varphi-\scal{\cdot}{v}$ is coercive. 
\end{enumerate}
Then 
\begin{equation}
\label{eIidfwetf548b9-23y}
\partial\widetilde{\varphi}(\eta,y)=
\begin{cases}
\menge{\big(\varphi(y/\eta)-\scal{y}{u}/\eta,u\big)}
{u\in\partial\varphi(y/\eta)},&\text{if}\;\:\eta>0;\\
C,&\text{if}\;\:\eta=0\;\text{and}\;y=0;\\
\emp,&\text{otherwise.}
\end{cases}
\end{equation}
\end{corollary}
\begin{proof}
In view of 
Proposition~\ref{pkJhy64989d21}\ref{pkJhy64989d21iii+}, it
suffices to suppose that $y\neq 0$ and to show that 
\begin{equation}
\label{eIidfwetf548b9-24a}
D=\menge{(\mu,u)\in\RR\times\GG}
{\mu+\varphi^*(u)\leq 0\;\:\text{and}\;\:
\sigma_{\dom\varphi^*}(y)=\scal{u}{y}}=\emp.
\end{equation}
Now denote by $\spts\dom\varphi^*$ the set of support points of
$\dom\varphi^*$. Then 
\begin{equation}
\label{eIidfwetf548b9-24b}
D=\menge{(\mu,u)\in\RR\times(\spts\dom\varphi^*)}
{\mu+\varphi^*(u)\leq 0\;\:\text{and}\;\:
\sigma_{\dom\varphi^*}(y)=\scal{u}{y}}.
\end{equation}

\ref{cIidfwetf548b9-23i}: We have $\sigma_{\dom\varphi^*}(y)=\pinf$
and therefore \eqref{eIidfwetf548b9-24a} yields $D=\emp$.

\ref{cIidfwetf548b9-23ii}: We have $\spts\dom\varphi^*=\emp$ and
therefore \eqref{eIidfwetf548b9-24b} yields $D=\emp$.

\ref{cIidfwetf548b9-23iii}$\Rightarrow$\ref{cIidfwetf548b9-23ii}:
Clear.

\ref{cIidfwetf548b9-23iv}$\Rightarrow$\ref{cIidfwetf548b9-23iii}:
\cite[Proposition~14.15]{Livre1}.

\ref{cIidfwetf548b9-23v}$\Rightarrow$\ref{cIidfwetf548b9-23iii}:
Let $v\in\GG$. Then by the Moreau-Rockafellar theorem
\cite[Theorem~14.17]{Livre1}, $\varphi-\scal{\cdot}{v}$ 
is coercive if and only if $v\in\intdom\varphi^*$. Hence
$\GG\subset\intdom\varphi^*$.
\end{proof}

Next, we provide an example of a perspective function
$g\in\Gamma_0(\RR\oplus\GG)$ such that $g\big|_{\dom g}$ 
is discontinuous.

\begin{example}
\label{exIidfwetf548b8-06}
Suppose that $\GG\neq\{0\}$, let $p\in\left]1,\pinf\right[$, and
set
\begin{equation}
\label{eIidfwetf548b1-09a}
g\colon\RR\oplus\GG\to\RX\colon (\eta,y)\mapsto
\begin{cases}
\|y\|^p/\eta^{p-1},&\text{if}\;\:\eta>0;\\
0,&\text{if}\;\:\eta=0\;\;\text{and}\;\;y=0;\\
\pinf,&\text{otherwise.}
\end{cases}
\end{equation}
Then $g\in\Gamma_0(\RR\oplus\GG)$ and 
$g\big|_{\dom g}$ is not continuous at $(0,0)$. 
Indeed, set $\varphi=\|\cdot\|^p$. Then $\varphi$ is a supercoercive
function in $\Gamma_0(\GG)$, and it thus follows from 
\eqref{ebcBdq739jU9-23a} that $\rec\varphi=\iota_{\{0\}}$.
Hence \eqref{eIidfwetf548b1-09a} coincides with 
\eqref{ekJhy64989d08l} and the first claim is therefore 
an application of 
Proposition~\ref{pkJhy64989d21}\ref{pkJhy64989d21ii}
with $g=\widetilde{\varphi}$.
Now set $y=(0,0)\in\RR\times\GG$, let $v\in\GG$ be such that 
$\|v\|=1$, fix a sequence $(\alpha_n)_{n\in\NN}$ in $\RPP$ such that
$\alpha_n\downarrow 0$, and set $(\forall n\in\NN)$ 
$y_n=(\alpha_n^{p/(p-1)},\alpha_nv)$. 
Then $(y_n)_{n\in\NN}$ lies in $\dom g$ and 
$y_n\to y$, but $\lim g(y_n)=1\neq 0=g(y)$. 
\end{example}

We now turn to some algebraic properties.

\begin{proposition}
\label{pkJhy64989d21'}
Let $\varphi\in\Gamma_0(\GG)$. Then the following hold:
\begin{enumerate}
\item
\label{pkJhy64989d21'iv}
Let $\psi\in\Gamma_0(\GG)$ be such that
$\dom\varphi\cap\dom\psi\neq\emp$, and let $\lambda\in\RPP$. 
Then $[\lambda\varphi+\psi]^\sim=\lambda\widetilde{\varphi}+
\widetilde{\psi}\in\Gamma_0(\RR\oplus\GG)$.
\item
\label{pkJhy64989d21'v}
Let $\Lambda\colon\HH\to\GG$ be linear, bounded, and such that 
$\ran\Lambda\cap\dom\varphi\neq\emp$. Set
$\widetilde{\Lambda}\colon\RR\oplus\HH\to\RR\oplus\GG\colon 
(\xi,x)\mapsto(\xi,\Lambda x)$.
Then $[\varphi\circ\Lambda]^\sim=\widetilde{\varphi}\circ
\widetilde{\Lambda}\in\Gamma_0(\RR\oplus\HH)$.
\item
\label{pkJhy64989d21'vi}
Suppose that $\varphi$ is positively homogeneous with
$\dom\varphi=\GG$, let $\phi\in\Gamma_0(\RR)$ be increasing on 
$\ran\varphi$ and such that $0\in\dom\phi$, let $\eta\in\RR$, 
and let $y\in\GG$.
Then $[\phi\circ\varphi]^\sim\in\Gamma_0(\RR\oplus\GG)$ and
$[\phi\circ\varphi]^\sim(\eta,y)=\widetilde{\phi}(\eta,\varphi(y))$.
\end{enumerate}
\end{proposition}
\begin{proof}
\ref{pkJhy64989d21'iv}: We have $\dom(\varphi+\psi)\neq\emp$. 
Hence $\varphi+\psi\in\Gamma_0(\GG)$ and 
\eqref{ebcBdq739jU9-23a} implies that 
$\rec(\lambda\varphi+\psi)=\lambda\rec\varphi+\rec\psi$. The
claim therefore follows from \eqref{ekJhy64989d08l} and 
Proposition~\ref{pkJhy64989d21}\ref{pkJhy64989d21ii}.

\ref{pkJhy64989d21'v}: Let $\xi\in\RR$ and $x\in\HH$. If
$\xi>0$, then $[\varphi\circ\Lambda]^\sim(\xi,x)=
\xi(\varphi\circ \Lambda)(x/\xi)=\xi\varphi(\Lambda x/\xi)
=(\widetilde{\varphi}\circ\widetilde{\Lambda})(\xi,x)$. Furthermore,
we have $\dom(\varphi\circ\Lambda)\neq\emp$. Hence,
$\varphi\circ\Lambda\in\Gamma_0(\HH)$ and 
\eqref{ebcBdq739jU9-23a} yields
$\rec(\varphi\circ\Lambda)=(\rec\varphi)\circ\Lambda$. Hence,
we derive from \eqref{ekJhy64989d08l} that
\begin{equation}
[\varphi\circ\Lambda]^\sim(0,x)=
\rec(\varphi\circ\Lambda)(x)=(\rec\varphi)(\Lambda x)=
(\widetilde{\varphi}\circ\widetilde{\Lambda})(0,x). 
\end{equation}
Finally, if $\xi<0$, then
$[\varphi\circ\Lambda]^\sim(\xi,x)=\pinf=
(\widetilde{\varphi}\circ\widetilde{\Lambda})(\xi,x)$.
Altogether, the conclusion follows from 
Proposition~\ref{pkJhy64989d21}\ref{pkJhy64989d21ii}.

\ref{pkJhy64989d21'vi}: The assumptions imply that $\varphi$ 
is continuous and that $\varphi(0)=0$. In turn 
$\phi\circ\varphi$ is lower semicontinuous and 
$0\in\dom(\phi\circ\varphi)$. 
It also follows from the assumptions that
$\phi\circ\varphi$ is convex. Altogether, 
$\phi\circ\varphi\in\Gamma_0(\GG)$ and we deduce from 
Proposition~\ref{pkJhy64989d21}\ref{pkJhy64989d21ii} that
$[\phi\circ\varphi]^\sim\in\Gamma_0(\RR\oplus\GG)$. Now suppose
that $\eta>0$. Then 
\begin{equation}
[\phi\circ\varphi]^\sim(\eta,y)=
\eta\phi\big(\varphi(y/\eta)\big)=\eta\phi\big(\varphi(y)/\eta\big)=
\widetilde{\phi}\big(\eta,\varphi(y)\big). 
\end{equation}
Next, we observe that, since
$0\in\dom(\phi\circ\varphi)$ and $0\in\dom\phi$, 
\eqref{ekJhy64989d08l} and \eqref{ebcBdq739jU9-23a} yield
\begin{align}
[\phi\circ\varphi]^\sim(0,y)
&=\rec(\phi\circ\varphi)(y)\nonumber\\
&=\lim_{\alpha\to\pinf}\frac{(\phi\circ\varphi)(0+\alpha y)}{\alpha}
\nonumber\\
&=\lim_{\alpha\to\pinf}\frac{\phi\big(\varphi(\alpha y)\big)}
{\alpha}
\nonumber\\
&=\lim_{\alpha\to\pinf}\frac{\phi\big(0+\alpha\varphi(y)\big)}
{\alpha}\nonumber\\
&=(\rec\phi)\big(\varphi(y)\big)\nonumber\\
&=\widetilde{\phi}\big(0,\varphi(y)\big).
\end{align}
Finally, if $\eta<0$, then $[\phi\circ\varphi]^\sim(\eta,y)=\pinf=
\widetilde{\phi}(\eta,\varphi(y))$.
\end{proof}

\begin{corollary}
\label{cIidfwetf548b8-28}
Let $\psi\in\Gamma_0(\GG)$ and let $C$ be a closed convex subset 
of $\GG$ such that $C\cap\dom\psi\neq\emp$. Set
\begin{equation}
\label{eIidfwetf548b8-29a}
g\colon\RR\times\GG\to\RX\colon (\eta,y)\mapsto
\begin{cases}
\eta\psi(y/\eta),&\text{if}\;\:\eta>0\;\:\text{and}\;\:
y\in\eta(C\cap\dom\psi);\\
(\rec\psi)(y),&\text{if}\;\:\eta=0\;\;\text{and}\;\;y\in\rec C;\\
\pinf,&\text{otherwise}.
\end{cases}
\end{equation}
Then $g\in\Gamma_0(\RR\oplus\GG)$.
\end{corollary}
\begin{proof}
This is an application of 
Proposition~\ref{pkJhy64989d21'}\ref{pkJhy64989d21'iv} with
$\lambda=1$ and $\varphi=\iota_C$. Indeed, in this setting,
$\rec(\varphi+\psi)=\rec\iota_C+\rec\psi=
\iota_{\rec C}+\rec\psi$ and \eqref{eIidfwetf548b8-29a} yields
$g=[\iota_C+\psi]^\sim$.
\end{proof}

\begin{corollary}
\label{ckJhy64989d20}
Let $\varphi\in\Gamma_0(\GG)$, let $\psi\in\Gamma_0(\GG)$ be 
a positively homogeneous function such that 
$\dom\varphi\cap\dom\psi\neq\emp$, and let $\delta\in\RR$. Then 
$[\varphi+\psi+\delta]^\sim\in\Gamma_0(\RR\oplus\GG)$ and
\begin{equation}
\label{e:f92f}
(\forall\eta\in\RR)(\forall y\in\GG)\quad
[\varphi+\psi+\delta]^\sim(\eta,y)=
\widetilde{\varphi}(\eta,y)+\psi(y)+\delta\eta.
\end{equation}
\end{corollary}
\begin{proof}
This follows from \eqref{ekJhy64989d08l} and
Proposition~\ref{pkJhy64989d21'}\ref{pkJhy64989d21'iv} since 
$\rec(\varphi+\psi+\delta)=(\rec\varphi)+(\rec\psi)=
(\rec\varphi)+\psi$.
\end{proof}

\begin{corollary}
\label{exIidfwetf548b9-10}
Let $\varphi\in\Gamma_0(\GG)$. Then 
$(\forall(\zeta,\eta)\in\RR^2)(\forall y\in\GG)$
$\widetilde{\widetilde{\varphi}}(\zeta,\eta,y)=
\widetilde{\varphi}(\eta,y)$.
\end{corollary}
\begin{proof}
By Proposition~\ref{pkJhy64989d21}%
\ref{pkJhy64989d21i}--\ref{pkJhy64989d21ii},
$\widetilde{\varphi}$ is a positively homogeneous function in
$\Gamma_0(\RR\oplus\GG)$. Hence the claim follows from 
Corollary~\ref{ckJhy64989d20}.
\end{proof}

\begin{proposition}
\label{pkJhy64989d28}
Let $I$ be a finite set and let $\eta\in\RR$. 
For every $i\in I$, let $\GG_i$ be a real Hilbert space, let 
$\varphi_i\in\Gamma_0(\GG_i)$, and let $y_i\in\GG_i$. Set
$\bigoplus_{i\in I}\varphi_i\colon\bigoplus_{i\in I}
\GG_i\to\RX\colon(z_i)_{i\in I}\mapsto
\sum_{i\in I}\varphi_i(z_i)$.
Then
\begin{equation}
\label{ekJhy64989d28a}
\Bigg(\bigoplus_{i\in I}\varphi_i\Bigg)^{\sim}
\Big(\eta,(y_i)_{i\in I}\Big)=\Bigg(
\bigoplus_{i\in I}\widetilde{\varphi_i}\Bigg)
\Big((\eta,y_i)\Big)_{i\in I}.
\end{equation}
\end{proposition}
\begin{proof}
Suppose that $\eta>0$. Then
\begin{equation}
\label{ekJhy64989d28b}
\Bigg(\bigoplus_{i\in I}\varphi_i\Bigg)^{\sim}
\big(\eta,(y_i)_{i\in I}\big)=
\eta\Bigg(\bigoplus_{i\in I}\varphi_i\bigg)(y_i/\eta)_{i\in I}=
\sum_{i\in I}\eta\varphi_i(y_i/\eta)=
\Bigg(\bigoplus_{i\in I}\widetilde{\varphi_i}\Bigg)
\Big((\eta,y_i)\Big)_{i\in I}.
\end{equation}
Now suppose that $\eta=0$. Then \eqref{ebcBdq739jU9-23a} implies that
$\rec\bigoplus_{i\in I}\varphi_i=\bigoplus_{i\in I}\rec\varphi_i$
and \eqref{ekJhy64989d28a} follows. Finally, if $\eta<0$, then
both sides of \eqref{ekJhy64989d28a} are equal to $\pinf$.
\end{proof}

Perspective functions can be used to provide examples of
nonintuitive behaviors for minimizing sequences in optimization
problems. 

\begin{example}
\label{ex:We439HghyV08a}
Suppose that $\GG=\RR$. Then 
Proposition~\ref{pkJhy64989d21}\ref{pkJhy64989d21ii} 
asserts that the function
\begin{equation}
\label{e:We439HghyV08a}
g=\big[\,|\cdot|^2\big]^\sim\colon\RR^2\to\RX
\colon (\xi_1,\xi_2)\mapsto
\begin{cases}
\xi_2^2/\xi_1,&\text{if}\;\:\xi_1>0;\\
0,&\text{if}\;\:\xi_1=\xi_2=0;\\
\pinf,&\text{otherwise}
\end{cases}
\end{equation}
belongs to $\Gamma_0(\RR^2)$. 
Moreover, Argmin\,$g=\RP\times\{0\}$. 
Now let $p\in\left[1,\pinf\right[$ and set
$(\forall n\in\NN)$ $x_n=((n+1)^{p+2},n+1)$. 
Then $(x_n)_{n\in\NN}$ is a minimizing sequence of
$g$ since $g(x_n)-\min g(\RR^2)=1/(n+1)^p\downarrow 0$. 
However, $d_{\text{Argmin}\, g}(x_n)=n+1\uparrow\pinf$.
To sum up, 
\begin{equation}
g(x_n)-\min g(\RR^2)=O(1/n^p),\quad\text{while}\quad
(\forall x\in\text{Argmin}\,g)\quad\|x_n-x\|\uparrow\pinf.
\end{equation}
This illustrates the fact that, even if it induces a very good 
convergence rate of the objective values $(g(x_n))_{n\in\NN}$, 
a minimizing sequence $(x_n)_{n\in\NN}$ may have extremely 
poor properties in terms of actually approaching a solution 
to the underlying minimization problem.
\end{example}

We now describe constructions of lower semicontinuous convex 
functions based on perspective functions. The first result is 
based on the composition of the perspective of a convex function
with an affine operator.

\begin{proposition}
\label{p:5th}
Let $L\colon\HH\to\GG$ be linear and bounded, let 
$\varphi\in\Gamma_0(\GG)$, let $r\in\GG$, let $u\in\HH$, let 
$\rho\in\RR$, and set 
\begin{equation}
\label{e:perspective2}
f\colon\HH\to\RX\colon x\mapsto
\begin{cases}
\big(\scal{x}{u}-\rho\big)\varphi
\bigg(\Frac{Lx-r}{\scal{x}{u}-\rho}\bigg),
&\text{if}\;\:\scal{x}{u}>\rho;\\
(\rec\varphi)\big(Lx-r\big),
&\text{if}\;\:\scal{x}{u}=\rho;\\
\pinf,&\text{if}\;\:\scal{x}{u}<\rho.
\end{cases}
\end{equation}
Suppose that there exists $z\in\HH$ such that
$Lz\in r+(\scal{z}{u}-\rho)\dom\varphi$ and $\scal{z}{u}\geq\rho$,
and set $A\colon\HH\to\RR\oplus\GG\colon x\mapsto
(\scal{x}{u}-\rho,Lx-r)$. Then 
$f=\widetilde{\varphi}\circ A\in\Gamma_0(\HH)$. 
\end{proposition}
\begin{proof}
By construction, $A$ is a continuous affine operator,
while $\widetilde{\varphi}\in\Gamma_0(\RR\oplus\GG)$ by 
Proposition~\ref{pkJhy64989d21}\ref{pkJhy64989d21ii}. 
Therefore $f=\widetilde{\varphi}\circ A$ 
is lower semicontinuous and convex. Finally, to show that $f$ is
proper, suppose first that $\scal{z}{u}>\rho$. Then 
$(Lz-r)/(\scal{z}{u}-\rho)\in\dom\varphi$ and hence
$z\in\dom f$. On the other hand, if $\scal{z}{u}=\rho$, then
$Lz-r\in\{0\}$. In turn, $f(z)=(\rec\varphi)(0)=0$ and therefore
$z\in\dom f$.
\end{proof}

The next result involves the marginal of a perspective function
(see \cite{Akia16} for a special case in the context of game
theory). 

\begin{proposition}
\label{p:gaubert}
Let $\varphi\in\Gamma_0(\GG)$ and let $K$ be a nonempty closed
bounded interval in $\RP$. Define
\begin{equation}
\label{eIidfwetf548b9-11}
g\colon\GG\to\RR\colon y\mapsto
\inf_{\eta\in K}\widetilde{\varphi}(\eta,y).
\end{equation}
Then $g\in\Gamma_0(\GG)$. 
\end{proposition}
\begin{proof}
Proposition~\ref{pkJhy64989d21}\ref{pkJhy64989d21ii} asserts
that $\widetilde{\varphi}\in\Gamma_0(\RR\oplus\GG)$. In turn, it
follows from \cite[Proposition~8.26]{Livre1} that $g$ is convex 
and from \cite[Lemma~1.29]{Livre1} that it is lower 
semicontinuous and proper. 
\end{proof}

\section{Examples of perspective functions}
\label{sec:3}
Our first construction involves a difference of convex
functions.

\begin{corollary}
\label{cIidfwetf548b9-24}
Let $\psi\in\Gamma_0(\GG)$ and let 
$\env(\psi^*)\colon u\mapsto\inf_{v\in\GG}(\psi^*(v)+\|u-v\|^2/2)$
be the Moreau envelope of $\psi^*$. Set
\begin{equation}
\label{eIidfwetf548b9-24v}
g\colon\RR\times\GG\to\RX\colon (\eta,y)\mapsto
\begin{cases}
\dfrac{\|y\|^2}{2\eta}-\eta(\env\psi)(y/\eta),
&\text{if}\;\:\eta>0;\\
\sigma_{\dom\psi}(y),&\text{if}\;\:\eta=0;\\
\pinf,&\text{if}\;\:\eta<0.
\end{cases}
\end{equation}
Then $g=[\env(\psi^*)]^\sim\in\Gamma_0(\RR\oplus\GG)$.
\end{corollary}
\begin{proof}
Set $q=\|\cdot\|^2/2$ and $\varphi=q-\env\psi$, and let
$\infconv$ denote the infimal convolution operation.
It follows from Moreau's decomposition \cite{Mor62b} (see also
\cite[Theorem~14.3(i)]{Livre1}) that
$\varphi=\env(\psi^*)\in\Gamma_0(\GG)$. In addition,
from basic convex analysis, 
\begin{equation}
\varphi^*=(\psi^*\infconv q)^*=\psi^{**}+q=\psi+q 
\end{equation}
and therefore 
Lemma~\ref{l:pjlaurent}\ref{l:pjlaurentii} yields
\begin{equation}
\label{eIidfwetf548b9-25a}
\rec\varphi=\sigma_{\dom\varphi^*}=\sigma_{\dom \psi}.
\end{equation}
In view of \eqref{ekJhy64989d08l} and 
Proposition~\ref{pkJhy64989d21}\ref{pkJhy64989d21ii}, we
conclude that $g=\widetilde{\varphi}\in\Gamma_0(\RR\oplus\GG)$.
\end{proof}

\begin{example}[generalized Huber function]
\label{ex:ghuber}
Let $C$ be a nonempty closed convex subset of $\GG$ and let $P_C$
denote its projector. Upon setting $\psi=\iota_C$ in 
Corollary~\ref{cIidfwetf548b9-24}, we deduce that the function 
\begin{equation}
\label{eIidfwetf548b9-24u}
g\colon\RR\times\GG\to\RX\colon (\eta,y)\mapsto
\begin{cases}
\scal{y}{P_C(y/\eta)}-\dfrac{\eta\|P_C(y/\eta)\|^2}{2},
&\text{if}\;\:y\notin\eta C\;\;\text{and}\;\;\eta>0;\\[3mm]
\dfrac{\|y\|^2}{2\eta},
&\text{if}\;\:y\in\eta C\;\;\text{and}\;\;\eta>0;\\[3mm]
\sigma_C(y),&\text{if}\;\:\eta=0;\\
\pinf,&\text{if}\;\:\eta<0
\end{cases}
\end{equation}
is in $\Gamma_0(\RR\oplus\GG)$. More precisely, 
$g=\widetilde{\varphi}$, where
$\varphi=\env(\psi^*)=\env\sigma_C$.
Let us further specialize by
taking $C=B(0;\rho)$ for some  
$\rho\in\RPP$. Then \eqref{eIidfwetf548b9-24u} reduces to 
\begin{equation}
\label{eIidfwetf548b9-24c}
g\colon\RR\times\GG\to\RX\colon (\eta,y)\mapsto
\begin{cases}
\rho\|y\|-\dfrac{\eta\rho^2}{2},&\text{if}\;\:\|y\|>\eta\rho\;\;
\text{and}\;\;\eta>0;\\
\dfrac{\|y\|^2}{2\eta},&\text{if}\;\:\|y\|\leq\eta\rho\;\;
\text{and}\;\;\eta>0;\\
\rho\|y\|,&\text{if}\;\:\eta=0;\\
\pinf,&\text{if}\;\:\eta<0.
\end{cases}
\end{equation}
We infer from Corollary~\ref{cIidfwetf548b9-24} that
$g=\widetilde{\varphi}$, where 
$\varphi=\env(\rho\|\cdot\|)=q-d_C^2/2$, that is,
\begin{equation}
\label{eUoboO83bd7212}
\varphi\colon\GG\to\RX\colon y\mapsto
\begin{cases}
\rho\|y\|-\Frac{\rho^2}{2}, &\text{if}\;\:\|y\|>\rho;\\[+2mm]
\Frac{\|y\|^2}{2},&\text{if}\;\:\|y\|\leq\rho.
\end{cases}
\end{equation}
In particular, if $\GG=\RR$, then $\varphi$ is known as 
the Huber function. This function was introduced in \cite{Hube64} 
and it plays an important role in robust statistics and signal 
processing \cite{Hube09,Niko05}, while its perspective function 
appears implicitly in robust regression problems 
\cite{Hube09,Lamb11,Owen07}. The fact that the Huber function is
the Moreau envelope of the absolute value function can already be
found in \cite{Boug89}; see also \cite{Boug99}. On the other hand,
if we specialize the perspective function \eqref{eIidfwetf548b9-24c} 
to the case when $\GG=\RR$ and $\rho=1$, we obtain the function 
\begin{equation}
\label{e:OpwkKjb6613c}
g\colon\RR^2\to\RX\colon (\eta,y)\mapsto
\begin{cases}
|y|-\dfrac{\eta}{2},&\text{if}\;\:|y|>\eta\;\;
\text{and}\;\;\eta>0;\\
\dfrac{|y|^2}{2\eta},&\text{if}\;\:|y|\leq\eta\;\;
\text{and}\;\;\eta>0;\\
|y|,&\text{if}\;\:\eta=0;\\
\pinf,&\text{if}\;\:\eta<0,
\end{cases}
\end{equation}
which is used in computer vision \cite{Zach10}, where it is
called the bivariate Huber function.
\end{example}

We now consider a function that combines distance and support
functions.

\begin{example}[generalized Berhu function]
\label{exIidfwetf548b9-26}
Let $C$ and $D$ be nonempty closed convex subsets of $\GG$, 
and let $\rho\in\RPP$. Then the function
\begin{equation}
\label{eIidfwetf548b9-26v}
g\colon\RR\times\GG\to\RX\colon (\eta,y)\mapsto
\begin{cases}
\dfrac{\eta d_C^2(y/\eta)}{2\rho}+\sigma_D(y),
&\text{if}\;\:\eta>0\;\;\text{and}\;\;y\notin\eta C;\\
\sigma_D(y),
&\text{if}\;\:\eta>0\;\;\text{and}\;\;y\in\eta C;\\
\sigma_D(y),&\text{if}\;\:\eta=0\;\;\text{and}\;\;y\in\rec C;\\
\pinf,&\text{otherwise}
\end{cases}
\end{equation}
is in $\Gamma_0(\RR\oplus\GG)$. To show this, set $q=\|\cdot\|^2/2$,
$\varphi=d_C^2/(2\rho)$, and
$\psi=\sigma_D$. Then $\varphi\in\Gamma_0(\GG)$ and 
$\psi$ is a positively homogeneous function in $\Gamma_0(\GG)$
such that $0\in\dom\varphi\cap\dom\psi$.
Furthermore, since $\varphi=\iota_C\infconv(q/\rho)$, we have 
$\varphi^*=\iota^*_C+(q/\rho)^*=\iota^*_C+\rho q$ and
therefore $\dom\varphi^*=\dom\iota^*_C$. In turn,
Lemma~\ref{l:pjlaurent} yields
\begin{equation}
\rec\varphi=\sigma_{\dom\varphi^*}=\sigma_{\dom\iota_C^*}=
\rec{\iota_C}=\iota_{\rec C}. 
\end{equation}
Altogether,
\begin{equation}
\label{e:1234f847}
g=\bigg[\frac{d_C^2}{2\rho}+\sigma_D\bigg]^\sim
\end{equation}
and the claim follows from 
Corollary~\ref{ckJhy64989d20}. An especially interesting case is
obtained when $C=B(0;\rho)$ and $D=B(0;1)$. Then $\rec C=\{0\}$,
$\sigma_D=\|\cdot\|$, and \eqref{eIidfwetf548b9-26v} therefore becomes
\begin{equation}
\label{eIidfwetf548b9-26w}
g\colon\RR\times\GG\to\RX\colon (\eta,y)\mapsto
\begin{cases}
\dfrac{\|y\|^2+\rho^2\eta^2}{2\eta\rho},
&\text{if}\;\:\eta>0\;\;\text{and}\;\;\|y\|>\eta\rho;\\
\|y\|,
&\text{if}\;\:\eta>0\;\;\text{and}\;\;\|y\|\leq\eta\rho;\\
0,&\text{if}\;\:\eta=0\;\;\text{and}\;\;y=0;\\
\pinf,&\text{otherwise.}
\end{cases}
\end{equation}
As seen above, $g$ is the perspective function of 
\begin{equation}
\label{eIidfwetf548b9-26o}
\vartheta\colon\GG\to\RX\colon y\mapsto
\begin{cases}
\dfrac{\|y\|^2+\rho^2}{2\rho},&\text{if}\;\:\|y\|>\rho;\\
\|y\|,&\text{if}\;\:\|y\|\leq\rho.
\end{cases}
\end{equation}
In the special case when $\GG=\RR$, $\vartheta$ arises in 
mechanics \cite{Alib13,Bouc12} as well as in statistics 
\cite{Owen07}, where it is called the Berhu (or reverse Huber)
function. The reason for this terminology is that 
\eqref{eUoboO83bd7212} exhibits a quadratic behavior on 
$B(0;\rho)$ and a sublinear behavior outside, while 
\eqref{eIidfwetf548b9-26o}  exhibits a sublinear behavior on 
$B(0;\rho)$ and a quadratic behavior outside. Applications of 
the perspective of the Berhu function in robust regression can 
be found in \cite{Lamb16} and in \cite{Owen07}.
\end{example}

We now turn to a type of function that is used in support vector 
machines and in computer vision.

\begin{example}[generalized Vapnik loss function]
\label{ex:OpwkKjb6610}
Let $\varepsilon\in\RPP$. By applying 
Proposition~\ref{pkJhy64989d21'}\ref{pkJhy64989d21'vi}
to $\phi\colon t\mapsto\max\{|t|-\varepsilon,0\}$ and
$\varphi=\|\cdot\|$, we obtain that the function 
\begin{equation}
\label{e:OpwkKjb6610b}
g\colon\RR\times\GG\to\RX\colon (\eta,y)\mapsto
\begin{cases}
d_{B(0;\varepsilon\eta)}(y),&\text{if}\;\:\eta>0;\\
\|y\|,&\text{if}\;\:\eta=0;\\
\pinf,&\text{if}\;\:\eta<0.
\end{cases}
\end{equation}
is the perspective function of 
$\vartheta=\text{max}\{\|\cdot\|-\varepsilon,0\}\big]$ and that
it is in $\Gamma_0(\RR\oplus\GG)$.
A special case of this function appears in the context of computer
vision in \cite{Zach10}. When $\GG=\RR$, $\vartheta$ is known as 
Vapnik's $\varepsilon$-insensitive loss function and it is 
employed in the area of support vector machines \cite{Vapn00}.
\end{example}

Our next construction involves a mix of positively homogeneous 
and norm-like functions.

\begin{example}
\label{exIidfwetf548b9-29}
Let $\psi\colon\GG\to\RP$ be a proper, lower semicontinuous,
positively homogeneous convex function, let $\delta\in\RR$, 
let $\rho\in\RP$, let $p\in\left[1,\pinf\right[$, let $v\in\GG$,
and set
\begin{equation}
\label{eIidfwetf548b9-29v}
g\colon\RR\times\GG\to\RX\colon(\eta,y)\mapsto
\begin{cases}
\delta\eta+\scal{y}{v}+\big|\rho\eta^p+\psi^p(y)\big|^{1/p},
&\text{if}\;\:\eta\geq 0;\\
\pinf,&\text{if}\;\:\eta<0.
\end{cases}
\end{equation}
Then $g=[\delta+\scal{\cdot}{v}+|\rho+\psi^p|^{1/p}]^\sim
\in\Gamma_0(\RR\oplus\GG)$. 
Indeed, set $\varphi=\delta+\scal{\cdot}{v}+|\rho+\psi^p|^{1/p}$
and $\phi=\big|\rho+|\cdot|^p\big|^{1/p}$. Then $\rec\phi=|\cdot|$ 
and $\varphi=\delta+\scal{\cdot}{v}+\phi\circ\psi$.
Altogether, we derive from Corollary~\ref{ckJhy64989d20} and 
Proposition~\ref{pkJhy64989d21'}\ref{pkJhy64989d21'vi} that
$g=\widetilde{\varphi}\in\Gamma_0(\RR\oplus\GG)$. Let us now
consider some special cases of this perspective function.
\begin{enumerate}
\item
\label{exIidfwetf548b9-29i}
Set $\psi=\|\cdot\|$, $v=0$, and $p=2$. Then 
\eqref{eIidfwetf548b9-29v} leads to the perspective function
\begin{equation}
\label{eIidfwetf548b9-30a}
g\colon\RR\times\GG\to\RX\colon(\eta,y)\mapsto
\begin{cases}
\delta\eta+\sqrt{\rho\eta^2+\|y\|^2},&\text{if}\;\:\eta\geq 0;\\
\pinf,&\text{if}\;\:\eta<0.
\end{cases}
\end{equation}
In the case when $\GG=\RR$, $\rho=1$, and $\delta=-1$, this 
function shows up in computer vision \cite{Hart03}, where 
$g(\eta,\cdot)$ is called the pseudo-Huber function.
\item
\label{exIidfwetf548b9-29ii}
Let $D$ be a nonempty closed convex cone in $\GG$, let $v=0$,
let $\delta=0$, let $\rho=1$, and let $|||\cdot|||$ be a norm on
$\GG$. Set $\HH=\RR\oplus\GG$, $K=\RP\times D$, and 
$\psi=|||\cdot|||+\iota_D$. Define a norm on $\HH$ by
$|||\cdot|||_p\colon(\eta,y)\mapsto(|\eta|^p+|||y|||^p)^{1/p}$.
Then \eqref{eIidfwetf548b9-29v} yields $g=|||\cdot|||_p+\iota_K$,
i.e.,
\begin{equation}
\label{e:OpwkKjb6602a}
g\colon\HH\to\RX\colon z\mapsto
\begin{cases}
|||z|||_p,&\text{if}\;\:z\in K;\\
\pinf,&\text{if}\;\:z\notin K.
\end{cases}
\end{equation}
\item
\label{exIidfwetf548b9-29iii}
Consider the following setting in \ref{exIidfwetf548b9-29ii}:
$N\geq 2$ is an integer, $\GG=\RR^{N-1}$, 
$|||\cdot|||$ is the $\ell^p$ norm on
$\RR^{N-1}$, $D=\RP^{N-1}$, and $K=\RP^N$. Then, if $\|\cdot\|_p$
denotes the $\ell^p$ norm on $\RR^N$, the corresponding 
perspective function \eqref{eIidfwetf548b9-29v} is 
\begin{equation}
\label{eIidfwetf548b9-30b}
g\colon\RR^N\to\RX\colon z\mapsto
\begin{cases}
\|z\|_p,&\text{if}\;\:z\in\RP^N;\\
\pinf,&\text{if}\;\:z\notin\RP^N.
\end{cases}
\end{equation}
\item
\label{exIidfwetf548b9-29i-}
Set $\GG=\RR$, $\psi=|\cdot|$, $v=-1$, $\rho=1$, and $\delta=-1$. 
Then \eqref{eIidfwetf548b9-29v} yields the generalized Fischer-Burmeister 
function
\begin{equation}
g\colon\RR^2\to\RX\colon(\eta,y)\mapsto
\begin{cases}
-\eta-y+\big|\eta^p+|y|^p\big|^{1/p},&\text{if}\;\:\eta\geq 0;\\
\pinf,&\text{if}\;\:\eta<0,
\end{cases}
\end{equation}
which is used is nonlinear complementarity problems \cite{Chen06}.
The original Fischer-Burmeister function is obtained for $p=2$.
\end{enumerate}
\end{example}

The example below extends constructions found in 
robust estimation and in machine learning.

\begin{example}
\label{exkJhy64989d27}
Let $\phi\in\Gamma_0(\RR)$ be an even function, let $v\in\GG$, 
and let $\delta\in\RR$. Then $\phi$ in increasing on $\RP$ and
$0\in\dom\phi$. In turn, it follows from 
Corollary~\ref{ckJhy64989d20} and 
Proposition~\ref{pkJhy64989d21'}\ref{pkJhy64989d21'vi}
that the function 
\begin{equation}
\label{ekJhy64989d17A}
g\colon\RR\oplus\GG\to\RX\colon (\eta,y)\mapsto
\begin{cases}
\delta\eta+\scal{y}{v}+\eta\phi(\|y\|/\eta),&\text{if}\;\:\eta>0;\\
\scal{y}{v}+(\rec\phi)(\|y\|),
&\text{if}\;\:\eta=0;\\
\pinf,&\text{if}\;\:\eta<0
\end{cases}
\end{equation}
is in $\Gamma_0(\RR\oplus\GG)$. More precisely, 
$g=[\delta+\scal{\cdot}{v}+\phi\circ\|\cdot\|]^\sim$.
Now assume further that $\dom\phi^*=\RR$. 
Then \cite[Theorem~3.4]{Ccm101} 
implies that $\phi^{**}=\phi$ is supercoercive and, therefore,
that $\varphi$ is likewise. In turn, we derive from 
\eqref{ebcBdq739jU9-23a} that $\rec\varphi=\iota_{\{0\}}$, which
allows us to rewrite \eqref{ekJhy64989d17A} as
\begin{equation}
\label{ekJhy64989d18A}
g\colon\RR\oplus\GG\to\RX\colon (\eta,y)\mapsto
\begin{cases}
\delta\eta+\scal{y}{v}+\eta\phi(\|y\|/\eta),&\text{if}\;\:\eta>0;\\
0,&\text{if}\;\:\eta=0\;\;\text{and}\;\;y=0;\\
\pinf,&\text{otherwise}.
\end{cases}
\end{equation}
In particular, when $\GG=\RR$, $\phi=|\cdot|^2$, and $v=0$,
\eqref{ekJhy64989d18A} has been used in robust estimation 
\cite{Hube09} and in machine learning \cite{Micc13}. 
\end{example}

\begin{example}
Let $\rho\in\RPP$, let $p\in\left[1,\pinf\right[$, and set 
\begin{equation}
\label{ekKjw53N8Uy2-04a}
\begin{array}{rcl}
g\colon\RR\times\GG&\to&\RX\\[3mm]
(\eta,y)&\mapsto&
\begin{cases}
\dfrac{\rho\|y\|^p}{\eta^{p-1}}+p\eta\ln\eta
-\eta\ln\big(\eta^p+\rho\|y\|^p\big),&\text{if}\;\:\eta>0;\\
\rho\|y\|,&\text{if}\;\:\eta=0\;\text{and}\;p=1;\\
0,&\text{if}\;\:\eta=0,\;y=0,\;\text{and}\;p>1;\\
\pinf,&\text{otherwise.}
\end{cases}
\end{array}
\end{equation}
Upon invoking 
Proposition~\ref{pkJhy64989d21}\ref{pkJhy64989d21'vi} with
$\varphi=\|\cdot\|$ and 
\begin{equation}
\label{ekKjw53N8Uy2-04b}
\phi\colon\RR\to\RX\colon t\mapsto
\rho|t|^p-\ln\big(1+\rho|t|^p\big),
\end{equation}
we see that $g=[\phi\circ\varphi]^\sim\in\Gamma_0(\RR\oplus\GG)$.
For $p=1$, \eqref{ekKjw53N8Uy2-04b} arises in inverse problems 
\cite{Invp07}. For $\GG=\RR$ and $\rho=p=1$, \eqref{ekKjw53N8Uy2-04a} 
is closely related to the so-called ``fair'' function in robust 
statistics \cite[Section~6.4.5]{Reyw83}. 
For $\GG=\RR^N$ and $\rho=p=1$, \eqref{ekKjw53N8Uy2-04a} is used
in least-squares regularization \cite{Elad07}.
\end{example}

\begin{example}
Let $p\in\left[1,\pinf\right[$ and set 
\begin{equation}
\label{ekKjw53N8Uy2-04c}
\begin{array}{rcl}
g\colon\RR\times\GG&\to&\RX\\[3mm]
(\eta,y)&\mapsto&
\begin{cases}
p\eta\ln\eta-\eta\ln(\eta^p-\|y\|^p),
&\text{if}\;\:\eta>0\;\text{and}\;\|y\|<\eta;\\
0,&\text{if}\;\:\eta=0\;\text{and}\;y=0;\\
\pinf,&\text{otherwise.}
\end{cases}
\end{array}
\end{equation}
It follows from 
Proposition~\ref{pkJhy64989d21}\ref{pkJhy64989d21'vi} applied
to $\varphi=\|\cdot\|$ and 
\begin{equation}
\label{ekKjw53N8Uy2-11b}
\phi\colon\RR\to\RX\colon t\mapsto
\begin{cases}
-\ln\big(1-|t|^p\big),
&\text{if}\;\:|t|<1;\\
\pinf,&\text{if}\;\:|t|\geq 1
\end{cases}
\end{equation}
that $g=[\phi\circ\varphi]^\sim\in\Gamma_0(\RR\oplus\GG)$.
For $\GG=\RR^{N}$ and $p=2$, \eqref{ekKjw53N8Uy2-04c} is closely 
related to a standard barrier for the Lorentz cone 
$\menge{(y,\eta)\in\RR^{N+1}}{\|y\|\leq\eta}$
\cite[Proposition~5.4.3]{Nest94}.
\end{example}

Proposition~\ref{p:5th} is an effective device for constructing 
a lower semicontinuous convex function in $\Gamma_0(\HH)$ 
by composing a perspective function $\widetilde{\varphi}$,
for some $\varphi\in\Gamma_0(\GG)$, with a 
continuous affine operator $A\colon\HH\to\RR\oplus\GG$ and,
possibly, a suitable convexity preserving operation
(see also Proposition~\ref{pjkwrf78Fgs02}).
For instance, the generalized TREX estimator of \cite{Scda16b} 
hinges on a special case of the following example in 
Euclidean spaces.
\begin{example}
\label{ex:5th}
Let $L\colon\HH\to\GG$ be linear and bounded, let $|||\cdot|||$ be
a norm on $\GG$ such that, for some $\chi\in\RPP$,
$|||\cdot|||\geq\chi\|\cdot\|$, let $r\in\GG$, let $u\in\HH$, 
let $\rho\in\RR$, let $q\in\left]1,\pinf\right[$, and let 
$s\in\left[1,\pinf\right[$. Set
\begin{equation}
\label{eIidfwetf548b6-24t}
h\colon\HH\to\RX\colon x\mapsto
\begin{cases}
\Frac{|||Lx-r|||^{qs}}{|\scal{x}{u}-\rho|^{(q-1)s}},
&\text{if}\;\:\scal{x}{u}>\rho;\\[3mm]
0,&\text{if}\;\:Lx=r\;\;\text{and}\;\;\scal{x}{u}=\rho;\\
\pinf,&\text{otherwise.}
\end{cases}
\end{equation}
Then $h\in\Gamma_0(\HH)$. 
\end{example}
\begin{proof}
Set $\varphi=|||\cdot|||^q$. Then $\dom\varphi=\GG$. In addition, 
$\varphi(y)/\|y\|\geq\chi^q\|y\|^q/\|y\|\to\pinf$ as 
$\|y\|\to\pinf$ and therefore \eqref{ebcBdq739jU9-23a} implies that
$\rec\varphi=\iota_{\{0\}}$. Thus, \eqref{e:perspective2} becomes
\begin{equation}
\label{eIidfwetf548b6-24m}
f\colon\HH\to\RX\colon x\mapsto
\begin{cases}
\dfrac{|||Lx-r|||^q}{|\scal{x}{u}-\rho|^{q-1}},
&\text{if}\;\:\scal{x}{u}>\rho;\\[3mm]
0,&\text{if}\;\:Lx=r\;\;\text{and}\;\;\scal{x}{u}=\rho;\\
\pinf,&\text{otherwise,}
\end{cases}
\end{equation}
and Proposition~\ref{p:5th} asserts that $f\in\Gamma_0(\HH)$.
Now let $\phi=|\cdot|^s$ and set $\phi(\pinf)=\pinf$. Then 
$\phi$ is increasing on
$\RPX=\ran f$, continuous, and convex. Hence it follows 
from \cite[Proposition~II.8.4]{Choq64} and 
\cite[Proposition~8.19]{Livre1} that
$h=\phi\circ f\in\Gamma_0(\HH)$.
\end{proof}

\begin{example}
\label{ex:12th}
Let $(\Omega,{\EuScript F},\PP)$ be a probability space and
let $\HH=L^2(\Omega,{\EuScript F},\PP)$ be the associated
Hilbert space of square-integrable random variables.
Let $\varphi\in\Gamma_0(\HH)$ and set 
\begin{equation}
\label{e:perspective12}
f\colon\HH\to\RX\colon X\mapsto
\begin{cases}
\EE X\varphi
\bigg(\Frac{X}{\EE X}\bigg),
&\text{if}\;\:\EE X>0;\\
(\rec\varphi)(X),
&\text{if}\;\:\EE X=0;\\
\pinf,&\text{if}\;\:\EE X<0.
\end{cases}
\end{equation}
Then $f\in\Gamma_0(\HH)$. 
\end{example}
\begin{proof}
This is an application of Proposition~\ref{p:5th} with 
$\GG=\HH$, $L=\Id$, $\mu=\PP$, $u=1$ a.s., $r=0$ a.s., $z=0$ a.s.,
and $\rho=0$.
\end{proof}

\section{Integral functions}
\label{sec:5}

In this section we construct lower semicontinuous functions by
using as an integrand a perspective function. 
First, let us extend and formalize the divergence model
\eqref{ekJhy64989d21b}.

\begin{proposition}
\label{pjkwrf78Fgs04}
Let $(\Omega,{\EuScript F},\mu)$ be a measure space, let
${\mathsf G}$ be a separable real Hilbert space, and let 
$\varphi\in\Gamma_0(\mathsf{G})$. Set
$\HH=L^2((\Omega,{\EuScript F},\mu);\RR)$ and
$\GG=L^2((\Omega,{\EuScript F},\mu);\mathsf{G})$, and suppose
that one of the following holds:
\begin{enumerate}
\item
\label{pjkwrf78Fgs04i}
$\mu(\Omega)<\pinf$.
\item
\label{pjkwrf78Fgs04ii}
$\varphi\geq\varphi(0)=0$. 
\end{enumerate}
For every $x\in\HH$, set 
$\Omega_0(x)=\menge{\omega\in\Omega}{x(\omega)=0}$ and
$\Omega_+(x)=\menge{\omega\in\Omega}{x(\omega)>0}$.
Define 
\begin{multline}
\label{ejkwrf78Fgs02b}
\Phi\colon\HH\oplus\GG\to\RX\colon (x,y)\mapsto\\
\begin{cases}
\displaystyle{\int_{\Omega_0(x)}}
\big(\rec\varphi\big)\big(y(\omega)\big)\mu(d\omega)\!\!\!&+
\displaystyle{\int_{\Omega_+(x)}}x(\omega)
\varphi\bigg(\dfrac{y(\omega)}{x(\omega)}\bigg)\mu(d\omega),\\[5mm]
&\text{if}\;\;
\begin{cases}
x\geq 0\;\:\text{a.e.}\\
(\rec\varphi)(y)1_{\Omega_0(x)}+x\varphi(y/x)1_{\Omega_+(x)}\in 
L^1\big((\Omega,\mathcal{F},\mu);\RR\big);\\[3mm]
\end{cases}\\
\pinf,&\text{otherwise.}
\end{cases}
\end{multline}
Then $\Phi\in\Gamma_0(\HH\oplus\GG)$.
\end{proposition}
\begin{proof}
It follows from 
Proposition~\ref{pkJhy64989d21}\ref{pkJhy64989d21ii} 
that $\widetilde{\varphi}\in\Gamma_0(\RR\oplus\mathsf{G})$.
Furthermore, we derive from \eqref{ekJhy64989d08l} and 
\eqref{ejkwrf78Fgs02b} that
\begin{equation}
\label{e:245hg2}
(\forall x\in\HH)(\forall y\in\GG)\quad
\Phi(x,y)=\int_{\Omega}\widetilde{\varphi}
\big(x(\omega),y(\omega)\big)\mu(d\omega).
\end{equation}
In turn, \cite[Proposition~9.32]{Livre1} yields
$\Phi\in\Gamma_0(\HH\oplus\GG)$. 
\end{proof}

\begin{proposition}
\label{pjkwrf78Fgs02}
Let $\Omega$ be a nonempty open subset of $\RR^N$ and let $\HH$ be
the Sobolev space $H^1(\Omega)$, i.e.,
$\HH=\menge{x\in L^2(\Omega)}{\nabla x\in(L^2(\Omega))^N}$.
For every $x\in\HH$, set 
$\Omega_-(x)=\menge{t\in\Omega}{x(t)<0}$,
$\Omega_0(x)=\menge{t\in\Omega}{x(t)=0}$, and
$\Omega_+(x)=\menge{t\in\Omega}{x(t)>0}$.
Let $\varphi\in\Gamma_0(\RR^N)$ be such that
$\varphi\geq\varphi(0)=0$, and define
\begin{equation}
\label{ejkwrf78Fgs04}
\begin{array}{rcl}
f\colon\HH&\to&\RX\\[3mm]
x&\mapsto&
\begin{cases}
\displaystyle{\int_{\Omega_0(x)}}
\big(\rec\varphi\big)\big(\nabla x(t)\big)dt+
\displaystyle{\int_{\Omega_+(x)}}x(t)
\varphi\bigg(\dfrac{\nabla x(t)}{x(t)}\bigg)dt,
&\text{if}\;\:x\geq 0\;\:\text{a.e.};\\[3mm]
\pinf,&\text{otherwise.}
\end{cases}
\end{array}
\end{equation}
Then $f\in\Gamma_0(\HH)$. 
\end{proposition}
\begin{proof}
Set $\GG=(L^2(\Omega))^N$ and $\mathsf{G}=\RR^N$, define 
$\Phi$ as in \eqref{ejkwrf78Fgs02b}, where 
$(\Omega,\mathcal{F},\mu)$ is 
the standard Lebesgue measure space, and let 
$L\colon\HH\to\HH\oplus\GG\colon x\mapsto(x,\nabla x)$. Then
$\Phi\in\Gamma_0(\HH\oplus\GG)$ by 
Proposition~\ref{pjkwrf78Fgs04}\ref{pjkwrf78Fgs04ii}. 
On the other hand, since $\nabla\colon\HH\to\GG$ is bounded, 
$L$ is linear and continuous. Since $f(0)=0$, we 
conclude that $f=\Phi\circ L\in\Gamma_0(\HH)$.
\end{proof}

The next examples recover two classical functions that have been
used extensively in statistics (Fisher information) and in image
recovery (total variation).  

\begin{example}
\label{ex:i9}
Consider the setting of Proposition~\ref{pjkwrf78Fgs02}.
\begin{enumerate}
\item
By choosing the supercoercive function $\varphi=\|\cdot\|_2^2$, 
we infer that the Fisher information 
\begin{equation}
\label{ejkwrf78Fgs04b}
\begin{array}{rcl}
f\colon H^1(\Omega)&\to&\RX\\[3mm]
x&\mapsto&
\begin{cases}
\displaystyle{\int_{\Omega_+(x)}}
\dfrac{\|\nabla x(t)\|_2^2}{x(t)}dt,
&\text{if}\;\:
\begin{cases}
x\geq 0\;\:\text{a.e.}\\
[\,x=0\;\Rightarrow\;\nabla x=0\,]\;\text{a.e.};
\end{cases}\\
\pinf,&\text{otherwise}
\end{cases}
\end{array}
\end{equation}
is in $\Gamma_0(H^1(\Omega))$. The convexity properties of 
\eqref{e:fisher} over the subspace of strictly positive
1-dimensional smooth densities were apparently first discussed in 
\cite{Cohe68}. The convexity and lower semicontinuity properties
of extensions of the Fisher information, such as those used in 
\cite{Lion95} for $N=1$ and based on $\varphi=|\cdot|^p$, with 
$p>1$, or on higher order derivatives, can be obtained analogously.
\item
By choosing the positively homogeneous function 
$\varphi=\|\cdot\|_2$, we infer that the total variation function
\begin{equation}
\label{eIidfwetf548b8-11t}
\begin{array}{rcl}
f\colon H^1(\Omega)&\to&\RX\\[3mm]
x&\mapsto&
\begin{cases}
\displaystyle{\int_{\Omega}}
\|\nabla x(t)\|_2dt,
&\text{if}\;\:x\geq 0\;\:\text{a.e.};\\[3mm]
\pinf,&\text{otherwise}
\end{cases}
\end{array}
\end{equation}
is in $\Gamma_0(H^1(\Omega))$.
\end{enumerate}
\end{example}

We can also derive from Proposition~\ref{pjkwrf78Fgs04}
lower semicontinuous versions of 
a variety of standard divergences in the continuous and discrete 
cases. In the former, the underlying measure space is the Lebesgue
measure space. The latter is illustrated below.

\begin{example}
\label{exIidfwetf548b1-08a}
\index{$\phi$-divergence}
\index{Csisz\'ar $\phi$-divergence}
Let $N$ be a strictly positive integer,
set $I=\{1,\ldots,N\}$, and let 
$\phi\in\Gamma_0(\RR)$. For every $x=(\xi_i)_{i\in I}\in\RR^N$ and 
every $y=(\eta_i)_{i\in I}\in\RR^N$, set
$I_-(x)=\menge{i\in I}{\xi_i<0}$,
$I_0(x)=\menge{i\in I}{\xi_i=0}$,
$I_+(x)=\menge{i\in I}{\xi_i>0}$, and
\begin{equation}
\label{eIidfwetf548b1-08c}
\Phi(x,y)=
\begin{cases}
\Sum_{i\in I_0(x)}(\rec\phi)(\eta_i)
+\Sum_{i\in I_+(x)}\xi_i\phi(\eta_i/\xi_i),
&\text{if}\;\:I_-(x)=\emp;\\ 
\pinf,&\text{if}\;\:I_-(x)\neq\emp.
\end{cases}
\end{equation}
Then $\Phi\in\Gamma_0(\RR^{2N})$.
Indeed, this is a special case of 
Proposition~\ref{pjkwrf78Fgs04}\ref{pjkwrf78Fgs04i}, 
where $\Omega=I$, $\mathcal{F}=2^I$,
$\mu$ is the counting measure (hence $\HH=\GG=\RR^N$), 
$\varphi=\phi$, and $\mathsf{G}=\RR$. For instance, consider
\begin{equation}
\label{exUoboO83bd7206z}
\phi\colon\RR\to\RX\colon t\mapsto
\begin{cases}
t\ln t,&\text{if}\;\:t>0;\\
0,&\text{if}\;\:t=0;\\
\pinf,&\text{if}\;\:t<0.
\end{cases}
\end{equation}
Then $\rec\phi=\iota_{\{0\}}$ and, if we set 
$J(x,y)=\menge{i\in I}{(\xi_i=0\;\text{and}\;\eta_i\neq 0)
\;\text{or}\;(\xi_i>0\;\text{and}\;\eta_i<0)}$, 
\begin{equation}
\Phi(x,y)=
\begin{cases}
\Sum_{i\in I_+(x)\cap I_+(y)}\eta_i\ln(\eta_i/\xi_i),
&\text{if}\;\:I_-(x)\cup J(x,y)=\emp;\\ 
\pinf,&\text{otherwise}
\end{cases}
\end{equation}
is the Kullback-Leibler divergence between $x$ and $y$.
This notion is central in statistics and in information theory.
Another noteworthy family of discrete divergences is obtained by
replacing \eqref{exUoboO83bd7206z} by
\begin{equation}
\label{exIidfwetf548b9-28z}
\phi\colon\RR\to\RX\colon t\mapsto
\begin{cases}
\big|t^{1/p}-1\big|^p,&\text{if}\;\:t\geq 0;\\
\pinf,&\text{if}\;\:t<0,
\end{cases}
\qquad\text{where}\quad p\in\left[1,\pinf\right[.
\end{equation}
In this case $\rec\phi=\sigma_{\left]\minf,1\right]}$ and,
if we set 
$J(x,y)=\menge{i\in I}{\xi_i\geq 0\;\text{and}\;\eta_i<0}$, 
\eqref{eIidfwetf548b1-08c} becomes
\begin{equation}
\Phi(x,y)=
\begin{cases}
\Sum_{i\in I_0(x)\cap I_+(y)}\eta_i+
\Sum_{i\in I_+(x)\smallsetminus I_-(y)}\big
|\eta_i^{1/p}-\xi_i^{1/p}\big|^p,
&\text{if}\;\:I_-(x)\cup J(x,y)=\emp;\\ 
\pinf,&\text{otherwise.}
\end{cases}
\end{equation}
We recover the Kolmogorov variational divergence for $p=1$ and 
the Hellinger divergence for $p=2$.
\end{example}

\paragraph{Acknowledgement.} The work of 
P. L. Combettes was partially supported by the 
CNRS MASTODONS project under grant 2016TABASCO.

\end{document}